\documentclass[11pt]{amsart}
\usepackage{verbatim, amscd, epsfig,amsmath,amsthm,amstext,amssymb,hyperref,bbm 
}

\theoremstyle{plain}
\newtheorem*{theorem*}{Theorem} 
\newtheorem{theorem}{Theorem}[section]
\newtheorem{lemma}[theorem]{Lemma}
\newtheorem{cor}[theorem]{Corollary}

\newtheorem{rem}[theorem]{Remark}

\numberwithin{equation}{section}

\theoremstyle{definition}

\newtheorem{example}[theorem]{Example}

\renewcommand{\Re}{{\rm Re}\,}
\renewcommand{\Im}{{\rm Im}\,}

\newcommand{\R}{\mathbb{ R}}
\newcommand{\C}{\mathbb{ C}}

\renewcommand{\H}{\mathbb{ H}}

\renewcommand{\P}{\mathbb{ P}}

\newcommand{\CP}{\C\P}
\newcommand{\invers}{^{-1}}

\newcommand{\calD}{\mathcal{D}}
\DeclareMathOperator{\End}{End}

\DeclareMathOperator{\SU}{SU}

\DeclareMathOperator{\Gl}{GL}
\DeclareMathOperator{\gl}{gl}
\DeclareMathOperator{\tr}{tr}
\DeclareMathOperator{\im}{im}
\DeclareMathOperator{\Span}{span}

\DeclareMathOperator{\id}{id}
\DeclareMathOperator{\Ad}{Ad}
\DeclareMathOperator{\Gr}{Gr}

\newcommand{\zo}{^{(0,1)}}
\newcommand{\oz}{^{(1,0)}}

\newcommand{\trivial}[1]{\underline{\H}^{#1}}
\newcommand{\trivialC}[1]{\underline{\C}^{#1}}

\newcommand{\ttrivial}[1]{\widetilde{\underline{\H}}^{#1}}
\newcommand{\ttrivialC}[1]{\widetilde{\underline{\C}}^{#1}}

\begin{document}
\title[Darboux transforms and simple factor dressing]{Darboux transforms and simple factor dressing of constant mean curvature surfaces}
\author{F. E. Burstall, J. F.  Dorfmeister, K. Leschke, A. Quintino}

\address{F. E. Burstall, Department of Mathematical Sciences, University
  of Bath, Bath BA2 7AY, United Kingdom } \address{J. F. Dorfmeister,
  Zentrum Mathematik, Technische Universit\"at M\"unchen, D-85747
  Garching, Germany} \address{K. Leschke, Department of Mathematics,
  University of Leicester, University Road, Leicester LE1 7RH, United
  Kingdom} \address{A. Quintino, Centro de Matem\'atica e Aplica\c{c}\~{o}es
Fundamentais da  Universidade de Lisboa, Avenida Professor Gama Pinto, 2, 1649-003 Lisboa,
  Portugal }

\email{F.E.Burstall@maths.bath.ac.uk, dorfm@ma.tum-de,  k.leschke@le.ac.uk, \linebreak aurea@ptmat.fc.ul.pt} 

\thanks{Third author partially supported by DFG SPP 1154 ``Global Differential Geometry''}


\begin{abstract}
  We define a transformation on harmonic maps $N: M \to S^2$ from a
  Riemann surface $M$ into the 2--sphere which depends on a parameter
  $\mu\in\C_*$, the so--called $\mu$--Darboux transformation. In the
  case when the harmonic map $N$ is the Gauss map of a constant mean
  curvature surface $f: M \to \R^3$ and $\mu$ is real, the Darboux
  transformation of $-N$ is the Gauss map of a classical Darboux
  transform of $f$. More generally, for all parameter $\mu\in\C_*$ the
  transformation on the harmonic Gauss map of $f$ is induced by a
  (generalized) Darboux transformation on $f$. We show that this
  operation on harmonic maps coincides with simple factor dressing,
  and thus generalize results on classical Darboux transforms of
  constant mean curvature surfaces \cite{darboux_isothermic},
  \cite{fran_epos}, \cite{inoguchi_kobayashi}: every $\mu$--Darboux
  transform is a simple factor dressing, and vice versa.
 
\end{abstract}

\maketitle
\section{Introduction}

By the Ruh--Vilms Theorem \cite{ruh_vilms} a constant mean curvature
surface $f: M \to \R^3$ of a Riemann surface $M$ into $\R^3$ is
characterized by the harmonicity of its Gauss map $N: M \to S^2$. This
fact allows to use integrable system methods for constant mean
curvature surfaces: using the harmonic Gauss map one can introduce a
spectral parameter $\lambda$ to obtain the associated $\C_*$--family
$d_\lambda$ of flat connections on the trivial $\C^2$ bundle over
$M$. This family is unitary on the unit circle where it describes the
associated family of harmonic maps on the universal cover $\tilde M$
of $M$.  More generally, we will use families of flat connections to
construct new harmonic maps: we consider a $\C_*$ family of flat
connections $d_\lambda$ so that $d_{\lambda=1} = d$ is the trivial
connection and $d_\lambda$ satisfies a reality condition. Assuming
that the map $\lambda\to d\zo_\lambda$ (which sends $\lambda\in\C_*$
to the $(0,1)$--part of $d_\lambda$) can be extended to a map on
$\CP^1$ which is meromorphic in $\lambda$ with only a simple pole at
zero we see that $d_\lambda$ is of the form $d_\lambda= d +
(\lambda-1) \omega\oz + (\lambda\invers-1)\omega\zo$ with $ \omega\oz$
and $\omega\zo$ of type $(1,0)$ and $(0,1)$ respectively. If in
addition $\omega\oz$ is nilpotent then $d_\lambda$ is the associated
family of a harmonic map. From our point of view, dressing of a
harmonic map \cite{uhlenbeck, terng_uhlenbeck} is thus the gauge of
$d_\lambda$ by an appropriate dressing matrix $r_\lambda$: we give
conditions on $r_\lambda$ such that $\hat d_\lambda = r_\lambda \cdot
d_\lambda$ is the associated family of a harmonic map. Fixing
$\mu\in\C_*$, a simple factor dressing is then given by a dressing
matrix $r_\lambda$ which has a simple pole at $\bar\mu\invers$ if
$\mu\not\in S^1$, and which depends on the choice of a
$d_\mu$--parallel line subbundle of the trivial $\C^2$ bundle over the
universal cover $\tilde M$ of $M$. We emphasize that both the
associated family and the dressing are given by an operation on the
harmonic map, and only the Sym--Bobenko formula \cite{sym_bob} then
induces a transformation on constant mean curvature surfaces.  We
compare our definition of a simple factor dressing with the simple
factor dressing in \cite{dorfmeister_kilian} which is defined on the
frame of the constant mean curvature surface: indeed both
transformations agree up to a rigid motion.

In contrast to dressing, the classical Darboux transformation is
originally a transformation on the level of surfaces: geometrically,
two conformal immersions $f, \hat f: M \to \R^3$ form a Darboux pair
if there exists a sphere congruence enveloping both $f$ and $\hat
f$. In this case, both $f$ and $\hat f$ are isothermic. In particular,
the classical Darboux transformation can be applied to a constant mean
curvature surface $f: M \to \R^3$: the map $\hat f: M \to\R^3$ is a
classical Darboux transform of $f$ if and only if $T = \hat f-f$ is a
solution of a certain Riccati equation. However, $\hat f$ has constant
mean curvature only if $T$ additionally satisfies an initial condition.

In \cite{conformal_tori} the classical Darboux transformation is
generalized to a transformation on conformal immersions $f: M \to S^4$
by weakening the enveloping condition. It turns out that this Darboux
transformation is also a key ingredient for integrable systems methods
in surface theory: in the case when $M=T^2$ is a 2--torus the spectral
curve of a conformal torus $f: T^2\to S^2$ is essentially the set of
all Darboux transforms $\hat f: T^2\to S^4$ of $f$.  
In this paper, we are interested in a (local) transformation theory
for general constant mean curvature surfaces $f: M\to \R^3$, and thus
have to allow the Darboux transforms $\hat f: \tilde M \to \R^4$ to be
defined on the universal cover $\tilde M$ of $M$.  To preserve the
constant mean curvature property we will only consider so--called
$\mu$--Darboux transforms \cite{cmc}: for $\mu\in\C_*$ a
$\mu$--Darboux transform $\hat f$ of a constant mean curvature surface
$f$ is constructed by using a parallel section of $d_\mu$ where
$d_\lambda$ is the associated family of the Gauss map $N$ of $f$. In
this case, the difference $T =\hat f-f$ between $f$ and $\hat f$ also
satisfies a Riccati type equation which generalizes the aforementioned
equation. It turns out $\hat f$ is a classical Darboux transform if
and only if $\mu\in S^1\cup \R_*$. In all cases $T$  satisfies the
required initial condition to preserve the constant mean curvature
property: every $\mu$--Darboux transform of a constant mean curvature
surface has constant mean curvature.  Thus, the $\mu$--Darboux
transformation induces a transformation on the harmonic Gauss map.

We extend this latter transformation to a $\mu$--Darboux
transformation on harmonic maps $N: M \to S^2$: using the associated
family of flat connections of $N$ and a $d_\mu$--parallel section, we
present an algebraic operation to obtain a new harmonic map on the
universal cover $\tilde M$ of $M$. It turns out, that a simple factor
dressing of a harmonic map $N$ coincides with a $\mu$--Darboux
transform of $N$: the $d_\mu$--parallel bundle which is used to define
the simple factor dressing matrix $r_\lambda$ is spanned by the
$d_\mu$--parallel section which gives the $\mu$--Darboux
transform.  In particular, we obtain a generalization of results on
the classical Darboux transformation \cite{darboux_isothermic,
  fran_epos, inoguchi_kobayashi}: a $\mu$--Darboux transform of a
constant mean curvature surface $f: M \to \R^3$ is given by a
simple factor dressing of the parallel constant mean curvature surface
of $f$, and vice versa.

Many other surface classes are also linked to harmonicity, e.g.,
Hamiltonian stationary Lagrangians $f: M \to \C^2$. In this case, the
so--called left normal $N: M \to S^2$ of $f$ is harmonic, and we can
apply again both a simple factor dressing and the $\mu$--Darboux
transformation on the harmonic map $N$. The $\mu$--Darboux
transformation on the harmonic left normal is induced by a
transformation on the level of surfaces \cite{hsl}. In particular,
though there is no Sym--Bobenko formula for Hamiltonian stationary
Lagrangians, we now have an interpretation of simple factor dressing
of a harmonic left normal $N$ on the level of surfaces: the left
normal of a $\mu$--Darboux transform of $f$ is the simple factor
dressing of $-N$.  Moreover, in \cite{quintino} a dressing on
(constrained) Willmore surfaces is introduced and it is shown that the
simple factor dressing with simple pole at $\mu\in\R_* \cup S^1$
coincides with the Darboux transformation on the conformal Gauss map
of a (constrained) Willmore surface as defined in
\cite{coimbra}. Indeed, the latter transformation can be extended
\cite{willmore_harmonic} in a way similar to what we discuss in this
paper to the case when $\mu\in\C_*$, and we expect that the simple
factor dressing of a (constrained) Willmore surface $f: M \to S^4$ is
the $\mu$--Darboux transformation on the conformal Gauss map of $f$.

\section{Harmonic maps and family of flat connections}

We first recall the well--known link \cite{uhlenbeck}
between a harmonic map into the 2--sphere and a $\C_*$--family of flat
connections. In the following we will always identify the Euclidean
4-space $\R^4$ with the quaternions $\H$.  In particular, we identify
$\R^3=\Im\H$ with Euclidean product $\langle a, b \rangle = -\Re(ab)$,
$a, b\in\Im\H$, and $S^2 =\{n\in \Im\H\mid n^2=-1\}$. A map $N\colon M
\to S^2$ from a Riemann surface $M$ into the 2--sphere is harmonic if
it is a critical point of the energy functional, that is \cite{Klassiker}, if
\[
d*dN = N dN \wedge *dN\,.
\]
Here we write for a 1--form $\omega$
\[
*\omega(X) = \omega(J_MX)\,, \quad X\in TM\,,
\]
where $J_M$ denotes the complex structure of the Riemann surface $M$. In
other words, $*$ is the negative Hodge star operator.

We decompose $dN$ into $(1,0)$ and $(0,1)$--parts 
\[
(dN)' =\frac 12(dN - N*dN), \quad (dN)''=\frac 12(dN + N*dN)
\]
with respect to $N$.  The harmonicity condition is now expressed by
the condition that $(dN)'$, or equivalently $(dN)''$, is closed.

Every smooth map $N: M \to S^2$ induces a quaternionic linear endomorphism
$J\in\Gamma(\End(\trivial{}))$ on the trivial bundle $\trivial{} =
M\times \H$ over $M$ by setting
\[
J\phi = N \phi, \quad \phi\in\Gamma(\trivial{})\,.
\]
If we define the Hopf fields of the complex structure $J$ (with
respect to the flat connection $d$) as
\begin{equation*}
A = \frac {J(dJ) + *dJ}4 \quad \text{ and } \quad Q = \frac {J(dJ) - *dJ}4 
\end{equation*}
then the closedness of $(dN)'$ can be rephrased
in terms of the complex structure $J$:
\begin{lemma}
\label{lem:hopf field coclosed}
  A smooth map $N: M \to S^2$ is harmonic if and only if the Hopf
  field $A$ of the associated complex structure
  $J\in\Gamma(\End(\trivial{}))$  satisfies
\[
d*A=0\,.
\]
\end{lemma}
Note that the derivative of $J$ is given in terms of the Hopf fields
by
\begin{equation}
\label{eq: derivative of J}
dJ = 2(*Q -*A)\,,
\end{equation}
and therefore, the condition $d*A=0$ is equivalent to $d*Q=0$. Since
\begin{equation}
\label{eq: hopf dJ}
A = \frac 12*(dJ)' \quad \text{ and } \quad Q=-\frac 12*(dJ)''
\end{equation}
the Hopf fields both
anti--commute with the complex structure $J$ and have type $(1,0)$ and
$(0,1)$ with respect to $J$, that is
\begin{equation}
\label{eq:type of Hopf fields}
*A = J A = - AJ, \quad *Q = - J Q = Q J\,.
\end{equation}
In particular, by type considerations this implies
\begin{equation}
\label{eq:A wedge Q=0}
A \wedge Q = Q \wedge A =0\,.
\end{equation}

Moreover, if we decompose the trivial connection $d$ on $\trivial{}$
into $J$ commuting and anti--commuting parts $ d = d_+ + d_- $ then
\begin{equation}
\label{eq:d-}
d_-= A + Q
\end{equation}
where we used $d_-=\frac 12 J(dJ)$ and the equations (\ref{eq:
  derivative of J}) and (\ref{eq:type of Hopf fields}).

To introduce a spectral parameter $\lambda\in\C_*$ we consider $\H$ as a
complex $\C^2$ via the splitting $\H = \C + j\C$ with $\C
=\Span\{1,i\}$.  In other words, if we define the complex structure
$I$ by right--multiplication by $i\in\H$, then $\C^2$ can be
identified with $ (\H, I)$. Under this identification
$I\in\End_\C(\C^2)$ becomes a complex linear endomorphism. For
simplicity of notation we will use the same symbol for the
endomorphism $\lambda=a+Ib\in\End_\C(\C^2)$, $a, b\in\R$, and the complex
number $\lambda=a+ib\in\C$ since $I \phi = \phi i$ for $\phi\in\C^2$.

For $\lambda\in\C_*$ we define the complex connection
\begin{equation}
\label{eq:dlambda}
d_\lambda = d + (\lambda -1) A\oz + (\lambda\invers-1)A\zo
\end{equation}
on the trivial bundle $\trivialC 2 =(\trivial{}, I)$ over $M$ where
   \[
A\oz = \frac 12(A - I*A)\, \quad \text{and } \quad A\zo=\frac 12(A +I*A)
\]
are the $(1,0)$ and $(0,1)$ parts of the Hopf field $A$ with respect
to the complex structure $I$ on $\trivial{}$.  We denote by
\[
\Gamma(K\End_\C(\trivialC 2))=\{\omega\in\Omega^1(\End_\C(\trivialC 2))\mid *\omega= I\omega\}
\]
 and
\[
\Gamma(\bar K\End_\C(\trivialC 2))=\{\omega\in\Omega^1(\End_\C(\trivialC 2))\mid *\omega=-I\omega\}
\]

the 1--forms with values in the complex linear endomorphisms of type
$(1,0)$ and $(0,1)$, taken with respect to the complex structure $I$.  With
this notation we have
\[
A\oz\in\Gamma(K\End_\C(\trivialC 2)), \quad A\zo\in\Gamma(\bar K\End_\C(\trivialC 2))\,.
\]

If we denote by $E$ and $E^\perp = E j$ the $\pm i$ eigenspaces of the
complex structure $J$ on $\trivial{}$ respectively then the orthogonal
projections with respect to the splitting $\trivial{} = E \oplus
E^\perp$ are given by
\[
\pi_E = \frac 12(1-IJ), \quad \pi_{E^\perp} = \frac12(1+IJ)\,.
\]
Since $J$ is quaternionic linear $J$ commutes with $I$, and so does
$A$. Recalling (\ref{eq:type of Hopf fields}) that $A$ anti--commutes
with $J$, we see
\begin{equation}
\label{eq:hopf with projection}
A\oz = A\pi_{E^\perp} = \pi_E A \,, \quad \text{ and } \quad A\zo = A\pi_E =\pi_{E^\perp}A\,,
\end{equation}
 in particular $(A\oz)^2=(A\zo)^2=0$, and
\begin{equation}
\label{eq:im ker A oz}
\im A\oz \subset E \subset \ker A\oz, \quad \im A\zo \subset E^\perp \subset \ker A\zo\,.
\end{equation}

Since $E^\perp = Ej$ we have $\pi_{E}(\phi j) = (\pi_{E^\perp} \phi)j$
and we obtain
\begin{equation}
\label{eq:reality for hopf field}
A\oz(\phi j) = (A\zo \phi)j\,, \quad \phi\in \Gamma(\trivial{})\,.
\end{equation}
Moreover, $\lambda(\phi j) = (\bar\lambda \phi)j$ for $\lambda\in\C$
so that  the reality condition
\begin{equation}
\label{eq:reality}
d_\lambda(\phi j) = (d_{{\bar\lambda}\invers}\phi) j\,, \quad \phi\in\Gamma(\trivial{}), 
\end{equation}
holds for the complex connection $d_\lambda$. In particular,
$d_\lambda$ is a quaternionic connection if and only if $\lambda\in
S^1$.

To compute the curvature of $d_\lambda$ we first observe that $I$ commutes
with $J$, and thus also with $A$, since $J$ is quaternionic linear. Denoting by
\begin{equation}
\label{eq:connection form}
\alpha_\lambda = (\lambda-1)A\oz + (\lambda\invers-1)A\zo
\end{equation}
the connection form of $d_\lambda$ we see with (\ref{eq:hopf with
  projection}) that
\[
\alpha_\lambda \wedge \alpha_\lambda = (2-\lambda-\lambda\invers)A\wedge A\,.
\]
On the other hand, we write with (\ref{eq:type of Hopf fields})
\begin{equation}
\label{eq:A10 via *A}
A\oz=*A\frac{J-I}2, \quad 
A\zo=*A\frac{J+I}2\,,
\end{equation}
and recall (\ref{eq: derivative of J}) and (\ref{eq:A wedge Q=0}) to obtain
\[
d\alpha_\lambda = (d*A)\left((\lambda-1) \frac{J-I}2 + (\lambda\invers-1)\frac{J+I}2\right) +
(\lambda+\lambda\invers-2)*A\wedge *A\,.
\]
Since $A\wedge A = *A \wedge *A$ the curvature of $d_\lambda$ is therefore
given by
\[
R_\lambda = (d*A)\left((\lambda-1)\frac{J-I}2 + (\lambda\invers-1) \frac{J+I}2\right)\,,
\]
Lemma \ref{lem:hopf field coclosed} now yields  the familiar link between
harmonic maps and  $\C_*$--families of flat connections:
\begin{theorem}
\label{thm:family of flat connections}
  A smooth map $N: M \to S^2$ is harmonic if and only if the associated
  family of complex connections
\[
d_\lambda = d + (\lambda -1) A\oz + (\lambda\invers-1)A\zo
\]
on the trivial bundle $\trivialC{2}$ over $M$ is flat for all
$\lambda\in\C_*$.
\end{theorem}
\begin{rem}
  Up to gauge equivalence, $d_\lambda$ is the family of
  flat connections used by \cite{hitchin-harmonic} to construct the
  spectral curve of a harmonic torus in the 2--sphere, see
  \cite{cmc}.
\end{rem}

The family of flat connections induces a $S^1$--family of
harmonic maps. This family is given by $d_\lambda$--parallel sections
and thus is only defined on the universal cover $\tilde M$ of $M$.  We
denote by $\ttrivial{} = \tilde M\times \H$ and
$\ttrivialC 2 =\tilde M \times\C^2$ the trivial bundles over
$\tilde M$.

\begin{theorem}
  \label{thm: associated family}
 Let $N: M \to S^2$ be a harmonic map
  from a Riemann surface $M$ into the 2--sphere, $J$ the corresponding
  complex structure, and $d_\lambda$ the associated family of flat
  connections. For $\mu\in\C_*$ the Hopf field $A_\mu = \frac{J(d_\mu
    J) + *d_\mu J}4$ of $J$, taken with respect to the flat
  connection $d_\mu$, is co--closed with respect to $d_\mu$ for $\mu\in
  \C_*$, that is
\[
d_\mu *A_\mu=0\,.
\]

In particular, if $\varphi\in\Gamma(\widetilde{\trivial{}})$ is a
$d_\mu$--parallel section for $\mu\in S^1$, then $N_{\varphi} =
\varphi\invers N \varphi: \tilde M \to S^2$ is harmonic with respect
to $d$. Furthermore, denoting by $\Phi$ the endomorphism given by left
multiplication by $\varphi$, the associated $\C_*$ family of flat
connections of $N_\varphi$ is given by the gauge
\[
d_{\varphi,\lambda} = \Phi\invers \cdot d_{\lambda\mu}\,.
\]
\end{theorem}

\begin{proof}
  Write $d_\mu= d +\alpha_\mu$ with connection form $\alpha_\mu$ given
  by (\ref{eq:connection form}).  The equations (\ref{eq:type of Hopf
    fields}) and (\ref{eq:A10 via *A}) 
  show $*\alpha_\mu = J\alpha_\mu = - \alpha_\mu J$ where we also used
  that $[I,J]=0$.  From  type arguments we therefore obtain
\begin{equation}
\label{eq:type}
\alpha_\mu\wedge (dJ)'' = (dJ)''\wedge \alpha_\mu=0\,.
\end{equation}
Using $d_\mu J = dJ +[\alpha_\mu, J]$ we also see with $[I, J]=0$
\begin{equation}
\label{eq: dmuJ'}
(d_\mu J)' = (dJ)' +
  [\alpha_\mu, J]\,,
\end{equation}
that is, the Hopf field of $J$ with respect to $d_\mu$ is
\begin{equation}
\label{eq:Hopf associated family}
A_\mu = \mu A\oz + \mu\invers A\zo\,.
\end{equation}
Since $J$ is harmonic with respect to $d$ and $d \alpha_\mu +
\alpha_\mu\wedge\alpha_\mu=0$ by the flatness of $d_\mu$, the
equations (\ref{eq: dmuJ'}) and (\ref{eq:type}) give
\begin{equation*}
d_\mu(d_\mu J)' = d\big((dJ)' + [\alpha_\mu, J]\big) + \Big[\alpha_\mu\wedge 
\big((dJ)' + [\alpha_\mu, J]\big)\Big] =0\,.
\end{equation*}
This shows that $*A_\mu = -\frac 12(d_\mu J)'$ is closed with respect
to $d_\mu$.  For $\mu\in S^1$ the connection $d_\mu$ is quaternionic,
and is given by the gauge $d_\mu = \Phi \cdot d$ of $d$ by $\Phi$.
Furthermore, the complex structure of $N_\varphi$ is given by
$J_\varphi = \Phi\invers J \Phi$, and we obtain $dJ_\varphi =
\Ad(\Phi\invers)(d_\mu J)$ since $\Phi$ is parallel with respect to
$d_\mu$. Thus, $J_\varphi$ has Hopf field
\begin{equation*}
A_{\varphi} = \Phi\invers A_\mu \Phi\,, 
\end{equation*}
and $0 =d_\mu(*A_\mu\circ \Phi) = \Phi \circ(d*A_\varphi)$ shows that
$N_\varphi$ is harmonic with respect to $d$.  Finally, from
(\ref{eq:Hopf associated family}) we see that
\[
d_\mu + (\lambda-1) A_\mu\oz + (\lambda\invers-1)
A_\mu\zo = d + (\lambda\mu-1) A\oz + ((\lambda\mu)\invers-1) A\zo =
d_{\lambda\mu}\,,
\]
and gauging $d_{\varphi,\lambda} = d + (\lambda-1) A_\varphi\oz + (\lambda\invers-1) A_\varphi\zo$ by $\Phi$ we get
\[
\Phi \cdot d_{\varphi,\lambda} = 
  d_{\lambda\mu}\,.
\]

\end{proof}

\begin{rem} Since $d_\mu$ is quaternionic for $\mu\in S^1$, the
  section $\varphi$ is unique up to a quaternionic constant and thus
  $N_\varphi$ is uniquely given by $\mu$ up to an orthogonal map. The
  family $N_\varphi$ is called the \emph{associated family} of $N$.
\end{rem}

\section{Dressing of a harmonic map into the 2-sphere}
 
We have seen that a harmonic map from a Riemann surface $M$ into the
2--sphere gives rise to an associated $\C_*$--family of flat
connections $d_\lambda$ on the trivial $\C^2$ bundle over $M$ and the
associated family of harmonic maps on the universal cover $\tilde M$
of $M$. To construct new harmonic maps from given ones via dressing,
one uses again the family of flat connections: Gauging $d_\lambda$ by
an appropriate $\lambda$--dependent dressing matrix, one can
reconstruct a harmonic map from the new family of flat connections.

\newcommand{\AK}{A^{(1,0)}}
\newcommand{\AKbar}{A^{(0,1)}}
\newcommand{\hatAK}{\hat A^{(1,0)}}
\newcommand{\hatAKbar}{\hat A^{(0,1)}}
\newcommand{\oK}{\omega^{(1,0)}}
\newcommand{\oKbar}{\omega^{(0,1)}}

\begin{theorem}
\label{thm:family gives harmonic}  
Let $M$ be Riemann surface and assume that for every $\lambda\in\C_*$
the connection
\begin{equation}
\label{eq:general harmonic connection}
d_\lambda = d + (\lambda -1) \oK + (\lambda\invers-1)\oKbar\,, \quad \lambda\in\C_*\,,
\end{equation}
on $\trivialC{2}$ is flat where
\[
\oK\in\Gamma(K\End_\C(\trivialC 2)), \quad \oKbar\in\Gamma(\bar K\End_\C(\trivialC 2))
\]
are non--trivial endomorphism--valued 1--forms on the trivial bundle $\trivialC 2 =
(\trivial{},I)$ over $M$ of type $(1,0)$ and $(0,1)$ with respect to $I$.
If $\oK$ is in addition  nilpotent
\begin{equation}
(\oK)^2=0 \label{eq:nilpotent}
\end{equation}

and $\oK$ and $\oKbar$ satisfy the reality condition
\begin{equation}
\oK (\phi j)= (\oKbar\phi)j, \quad \phi\in\Gamma(\trivialC 2),   \label{eq:reality on forms}
\end{equation}

then $d_\lambda$ is the associated family of a harmonic map $N: M \to
S^2$.
\end{theorem}
\begin{proof}
The kernel of $\oK$ defines a line bundle $E=\ker \oK$ away from the
zeros of $\oK$. To show that $E$ extends smoothly across the zeros of
$\oK$, we equip the bundle $K\End_\C(\trivialC 2)$ with a complex
holomorphic structure $D$ such that $\oK$ is a holomorphic section in
$(K\End_\C(\trivialC 2), D)$.

First note that a section $\sigma\in\Gamma(\bar K K)$ can be
identified with the 2--form $\hat\sigma\in \Omega^2(M)$ by setting
$\hat\sigma(X,Y) = \sigma(X,Y) -\sigma(Y,X)$ for
$X,Y\in\Gamma(TM)$. Under this identification, a complex holomorphic
structure on the canonical bundle $K$ is a complex linear operator $D:
\Gamma(K) \to \Omega^2(M)$
satisfying the Leibniz rule
\[
D(\omega\lambda) = (D\omega)\lambda - \omega\wedge d\lambda
\]
for $\omega\in \Gamma(K)$ and $\lambda: M \to \C$.  In particular,
when tensoring the canonical bundle $K$ with $\End_\C(\trivialC 2)$,
we see that for every $\eta\in\Gamma(\bar K \End_\C(\trivialC 2)$ the
operator $D: \Gamma(K\End_\C(\trivialC 2)) \to
\Omega^2(\End_\C(\trivialC 2))$,
\[
D\omega = d\omega - [\eta\wedge\omega]\,, \quad \omega\in
\Gamma(K\End_\C(\trivialC 2))\,, 
\]
is a complex holomorphic structure on $K\End_\C(\trivialC 2)$.  

Since $d_\lambda$ is flat for all $\lambda\in\C_*$ and $d$ is the
trivial connection we have
\[
0 =R_\lambda = (\lambda-1)d\oK +  (\lambda\invers-1)d\oKbar + (\lambda\invers-1)(\lambda-1)[\oKbar\wedge \oK]
\]
where we used that $ (\oK)^2=(\oKbar)^2 = 0$ by (\ref{eq:nilpotent})
and (\ref{eq:reality on forms}).  In particular, taking the
$\lambda$--coefficient we have
\[
0 = d\oK - [\oKbar \wedge \oK]\,,
\]
that is, $\oK\in\Gamma(K\End_\C(\trivialC 2))$ is a holomorphic
section with respect to the holomorphic structure $D=d-\oKbar$. But
then $\ker \oK$ can be extended across the zeros of $\oK$ into a
holomorphic line bundle $E$ over $M$.

We now define a complex structure $J\in\Gamma(\End(\trivial{}))$ on
$\trivial{} = E\oplus E^\perp$ by setting $J|_E= I$ and $J|_{E^\perp} =-I$. Note that
$J$ is quaternionic linear since $E^\perp = Ej$. If we decompose $d =
d_+ + d_-$ into $J$ commuting and anti--commuting parts, then $d_- =
\frac 12J(dJ)$, and $E$ and $E^\perp$ are $d_+$ stable, whereas $d_-$
maps $E$ to $E^\perp$ and vice versa.  For $\phi\in\Gamma(E)$ we
have
\[
0= (d-\oKbar)(\oK\phi) = -  \oK \wedge (d\phi - \oKbar\phi)
\]
where we used that $\oK$ is holomorphic with respect to $d-\oKbar$ and
$E\subset \ker\oK$.  Decomposing $d\phi = (d\phi)\oz + (d\phi)\zo$
into $(1,0)$ and $(0,1)$--part with respect to $I$, we see by type
arguments that
\[
\oK\wedge (d\phi)\oz =0
\]
and $(d\phi)\zo -\oKbar\phi$ is a 1--form with values in $E$.  Now
(\ref{eq:nilpotent}) shows $\im \oK \subset E$ and the reality
condition (\ref{eq:reality on forms}) gives $\im \oKbar \subset
E^\perp$.  Since $E$ is $d_+$ stable 
we thus see  for $\phi\in\Gamma(E)$ that
\begin{equation}
\label{eq:a01=omega01}
0= \pi_{E^\perp} ((d \phi)\zo - \oKbar\phi) = (\pi_{E^\perp} (d_-\phi)\zo) - \oKbar \phi\,.
\end{equation}
Recall (\ref{eq:d-}) that $d_- = A + Q$ and observe that
\[
Q\zo \phi = \frac 12(Q + *Q I)\phi = 0
\]
since $J|_E = I$ and (\ref{eq:type of Hopf fields})
holds. Substituting into (\ref{eq:a01=omega01}) we see
\[
 \oKbar  = A\zo
\]
on $E$. Since $E\subset \ker\oK$ the reality condition
(\ref{eq:reality on forms}) shows $E^\perp \subset \ker \oKbar$ and
from (\ref{eq:im ker A oz}) we thus see that $\oKbar = A \zo$ on
$\trivialC 2=E \oplus E^\perp$.  Using the reality conditions
(\ref{eq:reality on forms}) and (\ref{eq:reality for hopf field}) we
also have $\oK = A\oz$ on $\trivialC 2$.  In other words, $d_\lambda$
is the associated family of complex connections (\ref{eq:dlambda}) of
the map $N: M \to S^2$ which is given by the complex structure $J$. In
particular, since $d_\lambda$ is flat for all $\lambda\in\C_*$ the map
$N$ is harmonic by Theorem \ref{thm:family of flat connections}.
\end{proof}
\begin{rem}
\label{rem:complex structure of family of flat connections}
Let $d_\lambda$ be a family of flat connections satisfying the
assumptions of Theorem \ref{thm:family gives harmonic}. From the
previous proof we see that the associated harmonic map $N$ of
$d_\lambda$ has complex structure $J$ with $J|_E = I$ and $J_{E^\perp}
= -I$. Here $E$ is the line bundle defined by the kernel of $\oK$.
\end{rem}

To obtain families of flat connections $d_\lambda$ of the form
(\ref{eq:general harmonic connection}) we observe:
\begin{lemma}
\label{lem:connection of correct type}
  Let $d_\lambda$, $\lambda\in\C_*$, be a family of connections on $\trivialC
  2$ satisfying
\begin{enumerate}
\item the reality condition (\ref{eq:reality}) 
\[
d_\lambda(\phi j) = (d_{\bar\lambda\invers} \phi)j \quad \text{ for } \quad \phi\in\Gamma(\trivialC 2)\,,
\] 
\item the $(0,1)$--part of $d_\lambda$ can be extended to a meromorphic map $\lambda \mapsto d^{(0,1)}_\lambda$ on $\CP^1$ which is holomorphic on $\C_*\cup \infty$
  and has a simple pole at $0$, and \\
\item $d_{\lambda=1}=d$.
\end{enumerate}
Then the family of connections $d_\lambda$ is of the form
\[
d_\lambda = d + (\lambda -1) \oK + (\lambda\invers-1)\oKbar
\]

 with $\oK \in \Gamma(K\End_\C(\trivialC 2))$ and $\oKbar\in\Gamma(\bar K \End_\C(\trivialC 2))$.

\end{lemma}
\begin{proof}
  Write $d_\lambda = d+ \omega_\lambda$ with connection form
  $\omega_\lambda\in\Omega^1(\End_\C(\trivialC 2))$.  The conditions
  \thetag{ii} and \thetag{iii} imply that the $(0,1)$--part of $d_\lambda$
  is given by
$ d^{(0,1)} + (\lambda\invers-1)\oKbar_\lambda$ with
$\oKbar_\lambda\in\Gamma(\bar K\End_\C(\trivialC 2))$ for $\lambda\in\CP^1$.
Now $\lambda\mapsto \oKbar_\lambda$ is holomorphic on the compact $\CP^1$ thus
$\oKbar= \oKbar_\lambda$ is independent of $\lambda$. Finally, by the reality
condition (\ref{eq:reality}) the $(1,0)$-part of $d_\lambda$ is given by
$d\oz + (\lambda-1)\oK$ with $\oK(\phi j) = (\oKbar \phi)j$ for
$\phi\in\Gamma(\trivialC 2)$.

\end{proof}

By combining the previous results we can construct new harmonic maps
by gauging the associated family of flat connections of a given
harmonic map $N: M \to S^2$ with an appropriate $\lambda$--dependent
map. From this point on we will assume that $N$ is non--trivial: if
$N$ is not a constant harmonic map, then we can assume without loss of
generality, by passing to $-N$ if necessary, that the associated
family (\ref{eq:dlambda}) of flat connections $d_\lambda\not=d$ is
non--trivial.

\begin{theorem}[Dressing]
\label{thm:dressing}
Let $N: M \to S^2$ be a non--trivial harmonic map from a Riemann
surface $M$ into the 2--sphere and $d_\lambda$ the associated family
(\ref{eq:dlambda}) of flat connections.  If $r_\lambda : M \to
\Gl(2,\C)$ is a smooth map into the regular matrices with
\begin{enumerate}
\item $r_1 =\id$ is the identity matrix for $\lambda=1$, \\
\item $\lambda \to r_\lambda$ is meromorphic on $\CP^1$ and holomorphic at $0$
  and $\infty$, \\
\item $r_\lambda$ satisfies the (generalized) reality condition
\[
(r_\lambda\phi)j \sigma_\lambda=r_{\bar\lambda\invers}(\phi j)\,,
\quad \phi\in \trivialC 2\,,
\]
with $\sigma_\lambda\in\C_*$, and \\
\item the map $\lambda\to \hat d_\lambda$ is holomorphic on $\C_*$
  where $\hat d_\lambda = r_\lambda \cdot d_\lambda$ is the connection
  obtained by gauging $d_\lambda$ by the dressing matrix $r_\lambda$,
\end{enumerate}

then $\hat d_\lambda$ is the associated family (\ref{eq:dlambda}) of a harmonic map $\hat N:
M \to S^2$.   The harmonic map $\hat N$ is called the
\emph{dressing} of $N$ by $r_\lambda$.
\end{theorem}

\begin{proof} 
  We first show that $\hat d_\lambda$ satisfies the assumptions of Lemma
  \ref{lem:connection of correct type}.  Since the reality condition
  \thetag{iii} gives $r_\lambda\invers(\phi j) =
  (r_{\bar\lambda\invers}\invers \phi)\sigma_\lambda j$ we obtain,
  together with $r_\lambda\cdot d_\lambda = r_\lambda\circ d_\lambda
  \circ r_\lambda\invers$ and the reality condition (\ref{eq:reality})
  for $d_\lambda$, that $\hat d_\lambda(\phi j)= (\hat d_{\bar\lambda\invers}\phi)j$.
 
  From \thetag{iv} we see that $\lambda\mapsto \hat d_\lambda^{(0,1)}$
  is holomorphic for $\lambda\in\C_*$, and \thetag{ii} shows that this
  map extends holomorphically at $\infty$ and has a simple pole at
  $\lambda=0$. Finally $\hat d_{\lambda=1} = d_{\lambda=1} = d$ by
  \thetag{i}, and we can apply Lemma \ref{lem:connection of correct
    type} to get
\[
\hat d_\lambda = d + (\lambda - 1) \oK +
  (\lambda\invers-1)\oKbar\,
\]
with $\oK\in\Gamma(K\End_\C(\trivialC 2)), \oKbar\in\Gamma(\bar
K \End_\C(\trivialC 2))$.  Since $\lim_{\lambda\to \infty}\lambda\invers d_\lambda= \AK$
and $r_\lambda$ is holomorphic at $\infty$ we see that
\begin{equation}
\label{eq:oK with rinfty}
\oK = \lim_{\lambda\to \infty} \lambda\invers \hat d_\lambda = \lim_{\lambda\to
  \infty} r_\lambda \cdot \lambda\invers d_\lambda = \Ad(r_\infty) \AK
\end{equation}
and, similarly, $\oKbar = \Ad(r_{0})\AKbar$.
In particular, this shows $(\oKbar)^2=(\oK)^2 =0$, and since $\hat
d_\lambda$ satisfies the reality condition (\ref{eq:reality}) we also have
the reality condition (\ref{eq:reality on forms}) for the 1--forms
$\oK$ and $\oKbar$.  Therefore, the family of flat connections $\hat d_\lambda$ defines with Theorem \ref{thm:family gives harmonic} a harmonic
map $\hat N: M \to S^2$. 
\end{proof}

\begin{rem}
\label{rem:dressing complex structure}

By Remark \ref{rem:complex structure of family of flat connections}
the associated complex structure $\hat J$ of the family of flat
connections $\hat d_\lambda$ is given by the quaternionic extension of
$\hat J|_{\hat E} = I$ where $\hat E$ is given by the kernel of
$\oK$. Equation (\ref{eq:oK with rinfty}) shows $\hat E \subset
\ker(\Ad(r_\infty)\AK)$, and recalling that the $+i$ eigenspace $E$ of
$J$ satisfies $E\subset \ker \AK$, with equality away from the zeros of
$\ker \AK$, we see
\[
\hat E
= r_\infty E\,.
\]
Thus for $\hat\phi = r_\infty\phi\in\hat E$ we obtain
\[
\hat J\hat \phi = r_\infty\phi i = r_\infty J\phi\,,
\]
in other words, the complex structure $\hat J$ is given by extending
$\hat J|_{\hat E} = \Ad(r_\infty)J|_{\hat E}$ quaternionically.

\end{rem}

A particular dressing is given by prescribing that $r_\lambda$ has
only a simple pole in $\lambda$ on $\C_*$:

\begin{example}[Simple factor dressing]
\label{ex: simple factor dressing}
Let $N: M \to S^2$ be a non--trivial harmonic map from a Riemann
surface into the 2--sphere with associated family $d_\lambda$ of flat
connections (\ref{eq:dlambda}).  Fix $\mu\in\C_*$ and let $M_\mu$ be a
$d_\mu$ parallel line subbundle of the trivial $\ttrivialC 2=\tilde M
\times\C^2$ bundle over the universal cover $\tilde M$ of $M$.
Denoting by $\pi_\mu$ and $\pi_\mu^\perp$ the projections on $M_\mu$
and $M_\mu^\perp$ with respect to the splitting $ \ttrivialC 2 = M_\mu
\oplus M_\mu^\perp$ we define
\[
r_\lambda =   \pi_\mu  \circ \gamma_\lambda + \pi_\mu^\perp
\]
where  $\gamma_\lambda $ is the complex linear endomorphism given by
\[
\gamma_\lambda =  \frac{1-\bar\mu\invers}{1-\mu} \frac{\lambda-\mu}{\lambda-\bar\mu\invers} \in \End_\C(\C^2)\,.
\]
Note that $r_\lambda=\id$ for $\mu\in S^1$ so that $r_\lambda$
trivially satisfies the conditions of Theorem
\ref{thm:dressing}. Therefore, we will from now on assume that
$\mu\not\in S^1$.  Since $ \pi_\mu( \varphi j) = (\pi_\mu^\perp
\varphi) j $ and $ \bar\gamma_\lambda\invers =
\gamma_{\bar\lambda\invers} $ we see that the reality condition
\thetag{iii} in Theorem \ref{thm:dressing}
\[
(r_\lambda\phi)j \gamma_{\bar \lambda\invers} = r_{\bar\lambda\invers}(\phi j)
\]
holds for $\phi\in\trivialC 2$. As we have seen before this implies
that the reality condition (\ref{eq:reality}) holds for $\hat
d_\lambda = r_\lambda\cdot d_\lambda$. Moreover, $\lambda \mapsto
r_\lambda$ is meromorphic on $\C_*$ with simple zero at $\lambda=\mu$
and simple pole at $\lambda=\bar\mu\invers$, and $\lambda\mapsto
r_\lambda$ is holomorphic at $0$ and $\infty$ which in particular
shows \thetag{ii} of Theorem \ref{thm:dressing}.

The condition \thetag{i} of Theorem \ref{thm:dressing}, that is
$r_1=\id$, trivially holds.  Therefore, it only remains to verify the
holomorphicity of $\hat d_\lambda$, see \thetag{iv} of Theorem
\ref{thm:dressing}. The only issue is at $\lambda=\mu$ and
$\lambda=\bar \mu\invers$ but, since $\hat d_\lambda$ satisfies the
reality condition (\ref{eq:reality}), it is enough to consider
$\lambda=\mu$. We express $d_\lambda$ in terms of $d_\mu$ as
\[
d_\lambda = d_\mu + (\lambda-\mu)\AK + (\lambda\invers-\mu\invers)\AKbar
\]
so that
\[
\hat d_\lambda = r_\lambda \cdot d_\mu + (\lambda-\mu)\Ad(r_\lambda)\AK
+ \frac{\mu-\lambda}{\mu\lambda}\Ad(r_\lambda)\AKbar\,.
\]
We observe that $\Ad(r_\lambda)$ has only a simple pole at $\mu$ which
shows that $ (\lambda-\mu)\Ad(r_\lambda)\AK +
\frac{\mu-\lambda}{\mu\lambda} \Ad(r_\lambda)\AKbar $ is holomorphic
at $\lambda =\mu$. Finally, we decompose $d_\mu$ with respect to the
splitting $\trivialC 2= M_\mu \oplus M_\mu^\perp$
\[
d_\mu= \calD + \beta
\]
into a differential operator $\calD$ which leaves $M_\mu$ and
$M_\mu^\perp$ parallel, and a tensor $\beta$ mapping $M_\mu $ to $
M_\mu^\perp$ and vice versa. By assumption $M_\mu$ is
$d_\mu$--parallel, and thus $\beta|_{M_\mu} = 0$ and
$\Ad(r_\lambda)\beta=\gamma_\lambda \beta$.

Since $M_\mu$ and $M_\mu^\perp$ are $\calD$ stable, we see
$r_\lambda\cdot \calD = \calD$ so that
\[
\lambda \mapsto r_\lambda \cdot d_\mu= \calD + \gamma_\lambda \beta
\]
is holomorphic in $\lambda=\mu$. This shows that $\lambda\mapsto \hat
d_\lambda$ is holomorphic in $\C_*$.

Therefore, $r_\lambda$ satisfies for all $\mu\in\C_*$ the assumptions
of Theorem \ref{thm:dressing}. In particular, we obtain for every
$\mu\in\C_*$ and every choice of $d_\mu$--parallel line bundle $M_\mu$
a harmonic map $\hat N: \tilde M \to S^2$, the \emph{simple factor
  dressing of} $N$ given by $\mu$ and $M_\mu$.  Note that for $\mu\in
S^1$ the simple factor dressing $\hat N=N$ is trivial.
\end{example}

\section{Simple factor dressing of constant mean curvature surfaces}

By the Ruh--Vilms Theorem \cite{ruh_vilms} an immersion $f: M \to
\R^3$ from a Riemann surface into $\R^3$ has constant mean curvature
if and only if its Gauss map $N: M \to S^2$ is harmonic.  In
particular, for a given constant mean curvature surface $f$ the Gauss
map $N$ gives a harmonic complex structure $J$ on the trivial
$\H$--bundle $\trivial{}$ and we have an associated $\C_*$--family of
flat complex connections (\ref{eq:dlambda})
\[
d_\lambda = d + (\lambda-1) A\oz + (\lambda\invers-1) A\zo
\]
on $\C^2=(\H, I)$ where $A = \frac 12 *(dJ)' $ is the Hopf field
(\ref{eq: hopf dJ}) of $J$.  Conversely, given the family of flat
connections of a harmonic map $N: M \to S^2$ one can reconstruct a
constant mean curvature surface $f$ by the Sym--Bobenko formula. We
briefly summarize this construction in our notation. Recall first that
the Gauss map $N: M \to S^2$ of a conformal immersion $f: M \to \R^3$
satisfies \cite{coimbra}
\[
*df = Ndf = - df N\,.
\]
The splitting of $dN = (dN)'+(dN)''$ into $(1,0)$ and $(0,1)$--part of
$dN$ with respect to $N$ is the decomposition of the shape operator
into trace and trace-free parts \cite{coimbra} so that
\begin{equation}
\label{eq: mean curvature}
(dN)' = - H df
\end{equation}
where $H$ is the mean curvature of $f$.  We will assume from now on
without loss of generality that constant mean curvature surfaces have
mean curvature $H=1$. In this case $A$ is the left multiplication by
$-\frac{*df}2$ since 
\begin{equation}
\label{eq:A=df}
-2*A\phi = (dJ)'\phi =(dN)'\phi =- df\phi, \quad
\phi\in\Gamma(\trivial{})\,.
\end{equation}
\quad

\begin{theorem}[Sym--Bobenko formula, \cite{sym_bob}]
\label{theorem:sym bobenko}
Let $f: M \to \R^3$ be a constant mean curvature surface with Gauss
map $N: M \to S^2$.  If $d_\lambda$ is the associated family
(\ref{eq:dlambda}) of complex flat connections of $N$ then $f$ is
locally given, up to translation, by
\begin{equation}
\label{eq: sym bobenko}
f =- 2 \left(\left.\frac{\partial}{\partial t} \varphi_{e^{it}}\right \vert_{t=0} \right)\varphi_{\lambda=1}\invers
\end{equation}
where $\varphi_\lambda\in\Gamma(\ttrivial{})$ are
$d_\lambda$--parallel sections on the universal cover $\tilde M$ of
$M$, depending smoothly on $\lambda=e^{it}\in S^1$. Conversely, every
non--trivial harmonic map $N: M\to S^2$ from a simply connected
Riemann surface $M$ into the 2--sphere gives by (\ref{eq: sym
  bobenko}) a constant mean curvature surface $f: \tilde M \to \R^3$
on the universal cover $\tilde M$ of $M$.
\end{theorem}
\begin{proof}
  Let $d_\lambda$ be the associated family of flat connections of a
  non--trivial harmonic map $N: M \to S^2$ and
  $\varphi_\lambda\in\Gamma(\trivial{})$ a smooth family of sections
  with $d_\lambda\varphi_\lambda=0$.  With $\lambda=e^{it}\in S^1$ we
  obtain
\[
d\left(\left.\frac{\partial}{\partial t} \varphi_\lambda
\right \vert_{t=0}\right) =
\left.\frac{\partial}{\partial t} d\varphi_{\lambda} \right\vert_{t=0} =
I(A\zo - A\oz)\varphi_1
=- *A\varphi_1
\,.
\]
Now, if $f$ is a constant mean curvature surface then its Gauss map
$N$ has $(dN)'\not=0$ and $2*A$ is the left multiplication by
$df$. Thus $f$ is given, up to translation, by (\ref{eq: sym
  bobenko}).

Conversely, if $N$ is a non--trivial harmonic map then we may assume
that $(dN)'\not=0$, and the above computation shows that the map $f$
defined by (\ref{eq: sym bobenko}) satisfies $df\varphi = 2*A\varphi$
where $df$ does not vanish identically.  By (\ref{eq:type of Hopf
  fields}) we have $*df = N df = -df N$ so that $N$ is the Gauss map
of the (branched) conformal immersion $f$ and $f$ has constant mean
curvature by the Ruh--Vilms theorem.
\end{proof}

In particular, the associated family and the dressing of a harmonic
map $N: M \to S^2$ also induce new constant mean curvature surfaces.

\begin{cor}
  Let $f: M \to \R^3$ be a constant mean curvature surface and let
  $d_\lambda$ be the associated family of its Gauss map $N$.  For
  $\mu\in S^1$ let $\varphi\in\Gamma(\ttrivial{})$ be a
  $d_\mu$--parallel section and $N_{\varphi}=\varphi\invers N \varphi$
  the associated harmonic map. If $\varphi_\lambda$ is a smooth
  $S^1$--family of $d_\lambda$--parallel sections with
  $\varphi_{\lambda=\mu} = \varphi$ then
\[
f_{\varphi} = - 2 \varphi\invers
\left(\left.\frac{\partial}{\partial t}
    \varphi_{e^{it}}\right\vert_{t=s}\right)
\]
is a constant mean curvature surface $f_{\varphi}: \tilde M \to \R^3$
with Gauss map $N_{\varphi}$ where $\mu=e^{is}\in S^1$.

\end{cor}
\begin{proof}
  Recall from Theorem \ref{thm: associated family} that
  $d_{\varphi,\lambda }=\Phi\invers \cdot d_{\lambda \mu}$ is the
  associated family of $N_\varphi$ where $\Phi$ is the endomorphism
  given by left multiplication by $\varphi$. In particular,
  $\varphi^\mu_\lambda:= \varphi\invers \varphi_{\lambda\mu}$ is
  $d_{\varphi,\lambda}$ parallel, and the Sym--Bobenko formula shows
\[
f_\varphi = -2 \left.\left(\frac{\partial}{\partial t}\varphi^\mu_{e^{it}}\right)\right|_{t=0} (\varphi^\mu_{\lambda=1})\invers = -2 \varphi\invers \left(\left.\frac{\partial}{\partial t}
    \varphi_{e^{it}}\right\vert_{t=s}\right)\,.
\]

\end{proof}

\begin{cor}
\label{eq:simple factor cmc surface}
Let $f: M \to \R^3$ be a constant mean curvature surface, and let $N:
M \to S^2$ be the Gauss map of $f$. Then the dressing $\hat N$ of $N$
by $r_\lambda$ is the Gauss map of the constant mean curvature surface
\[
\hat f =  f -2 \left( \left(\left.\frac{\partial}{\partial t} r_{e^{it}}\right)\right\vert_{t=0}(\varphi_{\lambda=1})\right) \varphi_{\lambda=1}\invers\,,
\]
where $\varphi_\lambda\in\Gamma(\ttrivial{})$ are $d_\lambda$--parallel
sections, depending smoothly on $\lambda=e^{it}\in S^1$.
\end{cor}
\begin{proof}
  First note that $\left(\left.\frac{\partial}{\partial t}
      r_{e^{it}}\right)\right\vert_{t=0}$ is in general not
  quaternionic linear.  We recall that the associated family of $\hat
  N$ is given by Theorem \ref{thm:dressing} by $\hat d_\lambda =
  r_\lambda \cdot d_\lambda$. Thus for a smooth $S^1$--family
  $\varphi_\lambda$ of $d_\lambda$--parallel sections we see that
  $\hat\varphi_\lambda = r_\lambda \varphi_\lambda$ is $\hat
  d_\lambda$--parallel. Now the Sym--Bobenko formula (\ref{eq: sym
    bobenko}) and $r_1=\id$ give the claim.
\end{proof}

Identifying $\R^4=\H$ with $\gl(2,\C)$--matrices of the form
$\left\{\begin{pmatrix} a & -\bar b\\ b & \bar a
  \end{pmatrix} \mid a, b\in\C\right\}
$
via 
\[
a_0 + ja_1 \mapsto \begin{pmatrix} a_0 & -\bar a_1\\ a_1 & \bar a_0
  \end{pmatrix}
\]
the inner product in $\R^3$ is given by $<v,w> = -\frac 12 \tr(vw)$
for $v,w\in\R^3$, and the coordinate frame of an immersion $f: M \to
\R^3$ is under this identification the unique (up to sign) smooth map
$F: \tilde M \to \SU(2,\C)$ with
\[
e^{-\frac u2} f_x = -iF\sigma_1 F\invers, \quad e^{-\frac u2} f_y = -iF\sigma_2 F\invers, \quad N = -iF\sigma_3 F\invers
\]
where $z = x+ iy $ is a conformal coordinate, $e^u$ is the induced
metric, and $\sigma_l$ are the Pauli--matrices 
\[
\sigma_1 =\begin{pmatrix} 0 &1\\1&0
\end{pmatrix}, \quad \sigma_2 = \begin{pmatrix} 0 &-i\\ i&0
\end{pmatrix}, \quad \text{ and } \quad \sigma_3= \begin{pmatrix} 1 &0\\ 0&-1
\end{pmatrix}\,;
\]
in particular,  $f_z =\frac 12(f_x - if_y)$ and $f_{\bar z} = \frac
12(f_x + if_y)$ are given by
\begin{eqnarray}
f_z &=& - i e^{\frac u2} F e_- F\invers \nonumber
\\
f_{\bar z} &=&   -i  e^{\frac u2} F e_+ F\invers \label{eq: df, N in frame} \\
 N &=&  -iF  \sigma_3F\invers\, \nonumber  
\end{eqnarray}
with $e_- =\begin{pmatrix} 0&0\\1&0
\end{pmatrix}$ and $e_+ =\begin{pmatrix} 0&1\\0&0
\end{pmatrix}$.  Since the metric and the mean curvature of an
immersion $f$ are given by $e^u = 2<f_z, f_{\bar z}>$ and
$H=2e^{-u}<f_{z\bar z}, N>$ respectively, the frame $F$ of a constant
mean curvature surface $f$ with mean curvature $H=1$ and Hopf
differential $Qdz$,
$Q =<f_{zz}, N>$,
satisfies with $<f_z,f_z> = <f_{\bar
  z}, f_{\bar z}>=0$ the equations
\begin{eqnarray*}
F\invers F_z &=& \begin{pmatrix} -\frac 14  u_z &  Q e^{-\frac u2}\\-\frac 12 e^{\frac u2} & \frac 14 u_z
\end{pmatrix} 
\\[.3cm]
F\invers F_{\bar z} &=& \begin{pmatrix} \frac 14  u_{\bar z} &\frac 12 e^{\frac u2} \\
 -  \bar Q e^{-\frac u2}& -\frac 14 u_{\bar z}
\end{pmatrix} 
\,.
\end{eqnarray*}
In particular, the Gauss--Codazzi equations for the constant mean
curvature surface $f$ are satisfied if and only if
\[
d^F := d + F\invers dF
\]
is a flat connection. Again, we can introduce the spectral parameter
$\lambda\in\C_*$:
\begin{lemma}
\label{lemma: lawson frame}
Let $f: M \to \R^3$ be a constant mean curvature surface with Gauss
map $N$ and coordinate frame $F$ and let $F_\lambda: \tilde M \to
\Gl(2,\C)$, $\lambda\in\C_*$, be an \emph{extended frame} of $f$, that
is a solution of
\begin{eqnarray}
F_\lambda\invers (F_\lambda)_z   &=& \begin{pmatrix} -\frac{u_z}4 &  Q e^{-\frac u2}\\ -\frac\lambda 2 e^{\frac u2} & \frac{u_z}4
\end{pmatrix}  \label{eq: U0} \\[.3cm]
F_\lambda\invers (F_\lambda)_{\bar z} &=& \begin{pmatrix} \frac{  u_{\bar z}}4 &\frac{\lambda\invers}2 e^{\frac u2} \\
 - \bar Q e^{-\frac u2}& -\frac{ u_{\bar z}}4
\end{pmatrix} \label{eq: V0}
\end{eqnarray}
with $F_{\lambda=1} = F$. Then $F_\lambda$ gives the associated family
(\ref{eq:dlambda}) of flat connections of the Gauss map $N$ of $f$ by
\begin{equation}
\label{eq:dlambda with frame}
 d_\lambda = F \cdot d^{F_\lambda}
\end{equation}
where $d^{F_\lambda} = d + F_\lambda\invers dF_\lambda$.
\end{lemma}
\begin{proof}
  Recalling (\ref{eq:A=df}) that $A$ is the left multiplication by
  $-\frac{*df}2$ and $d_\lambda = d + \alpha_\lambda$ with
  \[
\alpha_\lambda = (\lambda-1)A\oz + (\lambda\invers-1) A\zo
\]
we see from (\ref{eq: df, N in frame}) 
\[
\alpha_\lambda\oz = -\frac{\lambda-1}2 e^{\frac u2} F e_- F\invers dz\,,
\quad
\alpha_\lambda\zo = \frac{\lambda\invers-1}2 e^{\frac u2} F e_+ F\invers d\bar z\,.
\]
Putting $S_\lambda = F_\lambda F\invers$ we have
\[
S_\lambda\invers dS_\lambda = F(F_\lambda\invers dF_\lambda - F\invers dF) F\invers
\]
and using (\ref{eq: U0}) and (\ref{eq: V0}) we get $S_\lambda\invers dS_\lambda
=\alpha_\lambda$.  Thus $d_\lambda = S_\lambda\invers \cdot d$ which shows the claim.
  \end{proof}
  From the previous lemma we see that for a constant $v\in\trivialC
  2=\trivial{}$ the section $\varphi_\lambda = F F_\lambda\invers v$
  is $d_\lambda$--parallel, and we obtain from (\ref{eq: sym bobenko})
  \[
f = 2 \left(\left. \frac{\partial}{\partial t} F_{e^{it}}\right\vert_{t=0} F\invers\right)  + \rm{const.}
\]
Note that this coincides with the usual Sym--Bobenko formula for the
extended frame: writing $U_\lambda= F\invers_\lambda (F_\lambda)_z$ we see
\begin{eqnarray*}
\left(\left.\frac{\partial}{\partial t} F_{e^{it}}\right|_{t=0} F\invers\right)_z &=& \left(\left.\frac{\partial}{\partial t}(F_{e^{it}})_z\right|_{t=0}\right) F\invers - \left.\frac{\partial}{\partial t} F_{e^{it}}\right|_{t=0} F\invers F_z F\invers
\\
&=& F\left(\left.\frac{\partial}{\partial t} U_{e^{it}}\right|_{t=0}\right) F\invers \\
&=& \frac 12 f_z
\end{eqnarray*}
where we used (\ref{eq: U0}) and (\ref{eq: df, N in frame}). A similar
argument gives $\left(\left.\frac{\partial}{\partial t} F_{e^{it}}\right|_{t=0} F\invers\right)_{\bar z} = \frac 12
f_{\bar z}$.

We now connect the simple factor dressing on the extended frame
\cite{terng_uhlenbeck, dorfmeister_kilian} with the frame independent
definition in Example \ref{ex: simple factor dressing}.  We fix
$\mu\in\C_*$ and recall that the simple factor dressing matrix is
given by
\[
r_\lambda =   \pi_\mu  \circ \gamma_\lambda + \pi_\mu^\perp
\]
where $M_\mu$ is a $d_\mu$--parallel bundle, $\pi_\mu$ and
$\pi_\mu^\perp$ denote the projections onto $M_\mu$ and $M_\mu^\perp$
respectively, and $\gamma_\lambda $ is the complex linear endomorphism
given by
\[
\gamma_\lambda =  \frac{1-\bar\mu\invers}{1-\mu} \frac{\lambda-\mu}{\lambda-\bar\mu\invers} \in \End_\C(\C^2)\,.
\]
In particular, the simple factor dressing $\hat N$ of $N$ by
$r_\lambda$ has associated family of flat connections $\hat
d_\lambda=r_\lambda\cdot d_\lambda$ and gives by Corollary
\ref{eq:simple factor cmc surface} a constant mean curvature surface $\hat f$. We denote
the extended frame of $\hat f$ by $\hat F_\lambda$. Then Lemma \ref{lemma: lawson frame} shows
that 
\[
\hat d_\lambda = \hat F \cdot d^{\hat F_\lambda}
\]
with $\hat F = \hat F_{\lambda=1}$. Writing $\hat S_\lambda = \hat
F_\lambda \hat F\invers$, we thus have $r_\lambda \cdot d_\lambda =
\hat S\invers_\lambda \cdot d$, and $d_\lambda = F\cdot d^{F_\lambda}$ gives
\begin{equation}
\label{eq:hatSlambda}
\hat S_\lambda = s_\lambda \circ S_\lambda \circ r_\lambda\invers
\end{equation}
with $S_\lambda= F_\lambda F\invers$, and a $z$--independent
$s_\lambda$. On the other hand, \cite{dorfmeister_kilian} give a
simple factor dressing $\tilde f$ of a constant mean curvature surface
$f$ with extended frame $F_\lambda$: the extended frame $\tilde
F_\lambda$ of $\tilde f$ is given by
\begin{equation}
\label{eq:simple factor dressing}
h_\lambda \circ F_\lambda = \tilde F_\lambda \circ g_\lambda
\end{equation}
where $h_\lambda= \pi_{M_0 }\tau_\lambda + \pi_{M_0^\perp}$ is given
by the choice of a constant line $M_0\subset\trivialC 2$ and
\[
\tau_\lambda = \frac{\lambda-\mu}{\bar\mu(\lambda-\bar\mu\invers)}\in
\End_\C(\C^2)\,.
\]
Moreover, $g_\lambda$ is obtained from $h_\lambda$ by replacing the
constant line $M_0$ by the line bundle $F_\mu\invers M_0$ given by the
extended frame, that is
\[
g_\lambda= \pi_{F_\mu\invers M_0} \tau_\lambda + \pi_{(F_\mu\invers M_0)^\perp}\,.
\]
Putting $s_\lambda = h\invers h_\lambda$, $h=h_{\lambda=1}$, we get
\[
s_\lambda = \pi_{M_0}\gamma_\lambda + \pi_{M_0^\perp}
\]
since $\tau_1\invers \tau_\lambda = \gamma_\lambda$. By
(\ref{eq:dlambda with frame})  the line $M_\mu = F
F_\mu\invers M_0$ is $d_\mu$--parallel, and the simple factor dressing
matrix $r_\lambda$ of $M_\mu$ is given by
\[
r_\lambda = F g\invers g_\lambda F\invers\,.
\]
By definition of $S_\lambda=F_\lambda F\invers$ and $s_\lambda=
h\invers h_\lambda$ this shows with (\ref{eq:simple factor dressing})
\[
s_\lambda \circ S_\lambda = (h\invers \tilde F_\lambda) \circ (h\invers \tilde F)\invers \circ  r_\lambda\,.
\]
Plugging into (\ref{eq:hatSlambda}) we see that $\hat F_\lambda =
h\invers\tilde F_\lambda W$ where $W: M \to \Gl(2,\C)$ is
independent of $\lambda$.  The Sym--Bobenko formula then yields that $\tilde
f$ and $\hat f= h\invers \tilde f h$ coincide up to translation.

\section{Darboux transforms  of harmonic maps into the 2--sphere}

The classical Darboux transformation on isothermic surfaces can be
extended to a transformation on conformal maps $f: M \to S^4$ from a
Riemann surface into the 4-sphere, \cite{conformal_tori}. In the case
when $f$ is a constant mean curvature surface, one obtains a genuine
generalization of the classical Darboux transformation
\cite{cmc}. Here we consider a special case of the general Darboux
transformation, the so--called $\mu$--Darboux transforms. These have
constant mean curvature, but are only classical Darboux transforms for
special spectral parameter $\mu$. In particular, we obtain an induced
transformation on harmonic maps $N: M \to S^2$.

\begin{theorem}[\cite{cmc}]
\label{thm:mu-darboux}
  Let $f: M \to \R^3$ be a constant mean curvature surface in $\R^3$
  with Gauss map $N$ and associated family $d_\lambda$ of flat
  connections. For $\mu\in \C_*$ and $d_\mu$--parallel section
  $\varphi\in\Gamma(\ttrivial{})$, define
\[
 T = \frac 12(N \varphi(a-1)\varphi\invers + \varphi b\varphi\invers)
\]
where $a=\frac{\mu + \mu\invers}2, b = i \frac{\mu\invers-\mu}2$. Then
$ T$ is nowhere vanishing if $\mu\not=1$, and the map $\hat f: \tilde
M \to\R^4=\H$,
\[
\hat f = f + T\invers,
\] 
has constant real part. Moreover, $\im\hat f$ is a constant mean
curvature surface with Gauss map
\[
\hat N = - T\invers N T\,.
\]
The map $\hat f$ is called a \emph{$\mu$--Darboux transform} of $f$.
\end{theorem}

In other words, a $\mu$--Darboux transform is, up to a translation, a
constant mean curvature surface in $\R^3$.  Note that $\hat f$ depends
on the choice of the $d_\mu$--parallel section
$\varphi\in\Gamma(\ttrivial{})$.

The Darboux transformation is a key ingredient \cite{conformal_tori}
for integrable systems methods in surface theory. In the case when
$M=T^2$ is a 2--torus the spectral curve of a conformal torus $f:
T^2\to S^2$ is essentially the set of all Darboux transforms $\hat f:
T^2\to S^4$ of $f$.  If $f: T^2\to \R^3$ is a constant mean curvature torus
this general spectral curve is biholomorphic \cite{cmc} to the
spectral curve of the harmonic Gauss map $N$ of $f$: the spectral
curve of $N$ is given \cite{hitchin-harmonic} by the compactification
of the Riemann surface which is given by the eigenlines of the
holonomies of $d_\lambda$, $\lambda\in\C_*$. The eigenlines are
exactly given by parallel sections with multipliers, that is
$d_\mu$--parallel sections $\varphi$, $\mu\in\C_*$, of the trivial
$\C^2$ bundle over $\C$  which satisfy $\gamma^*\varphi =\varphi
h_\gamma$ with $h_\gamma\in\C_*$ for $\gamma\in\pi_1(T^2)$.  On the
other hand, parallel sections give $\mu$--Darboux transforms, and the
multiplier condition then implies that the $\mu$--Darboux transform is
a conformal map on the torus.

Here we are interested in (local) transformation theory of general
constant mean curvature surfaces $f: M \to \R^3$ and will allow the
$\mu$--Darboux transforms to be defined on the universal cover $\tilde
M$ of $M$.  Note that $a, b\in\C$ in the above theorem satisfy $a^2 +
b^2=1$, however, $a,b\in\R$ if and only if $\mu\in S^1$. In this case,
$ T$ is independent of the choice of the $d_\mu$--parallel section
$\varphi$ and $T\invers=N + \frac{b}{1-a}$. In particular, $\hat f = g
+ \frac{b}{1-a}$ is a translate of the parallel constant mean
curvature surface $g= f+N$ of $f$.  On the other hand, for
$\mu\in\R_*$ we see that $a\in\R$ so that $ T$, and thus $\hat f$,
takes values in $\R^3$.

\begin{theorem}[\cite{cmc}]
  Let $f: M \to \R^3$ be a constant mean curvature surface. Then a
  constant mean curvature surface $\hat f: \tilde M \to \R^4$ is a
  classical Darboux transform of $f$ if and only if $\hat f$ is a
  $\mu$--Darboux transform of $f$ with $\mu\in\R_*\cup S^1$.
\end{theorem}

By Theorem \ref{thm:mu-darboux} the $\mu$--Darboux transformation
preserves the harmonicity of the Gauss map. More generally:

\begin{theorem}
\label{thm: mu darboux for gauss map}
Let $N: M \to S^2$ be a non--trivial harmonic map from a Riemann
surface into the 2--sphere and $d_\lambda$ the associated family of
flat connections (\ref{eq:dlambda}). Define for $\mu\in\C_*$ and
$d_\mu$--parallel section $\varphi\in\Gamma(\ttrivial{})$ the map $ T:
\tilde M \to \H$
\begin{equation}
\label{eq: tilde T}
T = \frac 12(N \varphi(a-1)\varphi\invers + \varphi b\varphi\invers)
\end{equation}
where $a=\frac{\mu + \mu\invers}2, b = i \frac{\mu\invers-\mu}2$. Then
$ T$ is nowhere vanishing if $\mu\not=1$, and 
\[
\hat N =   T\invers N  T
\]
is harmonic. We call $\hat N$ a \emph{$\mu$--Darboux transform} of
$N$.
\end{theorem}
\begin{rem} Again, we emphasize that $\hat N$ depends in general on
  the choice of the $d_\mu$--parallel section $\varphi$. However, if
  $\mu\in S^1$, then $a, b\in\R$ and $ T$ is independent of
  $\varphi$. But then $[ T,N]=0$ gives $\hat N =N$ for $\mu\in
  S^1$. 

  In particular, our choice of sign for a $\mu$--Darboux transform is
  so that it coincides with the sign of the simple factor dressing of
  a harmonic map on $S^1$. However note that with this choice the
  $\mu$--Darboux transform of the Gauss map of a constant mean
  curvature surface $f$ is the negative Gauss map of the
  $\mu$--Darboux transform of $f$.

\end{rem}
\begin{proof}
  We essentially follow the proof in \cite{cmc} for the analogue
  statement for the Gauss map of a constant mean curvature surface.
  Putting $\hat z = \varphi z \varphi\invers$ for $z\in\C$ we write $
  T = \frac 12(N(\hat a-1) + \hat b)$ and, if $\mu\not=1$, then
\begin{equation}
\label{eq:initial condition}
2T(1-\hat a)\invers + N = \frac{\hat b}{1-\hat a}\,.
\end{equation}
Since $ N^2(p) =-1$ and $\frac{\hat b^2}{(1-\hat a)^2}= \frac{1+\hat
  a}{1-\hat a} \not=-1$ for all $\mu\in\C_*, \mu\not=1$, this shows
that $ T(p)\not=0$ for all $p\in M$. Next we observe
with (\ref{eq:dlambda}) and (\ref{eq:A10 via *A}) that
\begin{equation}
\label{eq:dmu with a, b}
d_\mu =d+ *A(J(a-1)+b)
\end{equation}
which shows  with (\ref{eq: hopf dJ}) that
\[
0= d_\mu \varphi= d \varphi - (dN)' T \varphi
\]
and thus $d\hat z = [(dN)'T, \hat z]$ for $z\in\C$. Differentiating
(\ref{eq: tilde T}) gives with $\hat a^2 + \hat b^2=1$ the Riccati
type equation
\begin{equation}
\label{eq:Riccati}
d T = (dN)''\frac{\hat a-1}2 -  T (dN)'  T\,,
\end{equation}
which shows
\[
NdT -*dT = -\left((dN)''(N(\hat a-1)+ (NT +TN)(dN)'T\right)\,.
\]
Since $d\hat N = [\hat N, T\invers d T] + T\invers dN T $ we thus
obtain
\[
d\hat N + \hat Nd*\hat N =\frac 12
T\invers(dN + N*dN)(\hat b -(\hat a-1)\hat N)\,.
\]
Now (\ref{eq: tilde T}) gives $-(\hat a-1) + T\hat N(\hat a-1)
-T\hat b=0$, that is
\[
\hat N = T\invers + \frac{\hat b}{\hat a-1}\,,
\]
and using the Riccati type equation (\ref{eq:Riccati}) we obtain
\[
d*\hat Q = d *A
\]
for the Hopf fields of $\hat N$ and $N$. This shows that $\hat N$
is harmonic.
\end{proof}

 Note that for $\mu\in\R_*$ the equation (\ref{eq:Riccati})
  is independent of the choice of the parallel section $\varphi$. In
  particular, if $N$ is the Gauss map of a constant mean curvature
  surface $f: M \to\R^3$ then the solutions of the Riccati equation
  (\ref{eq:Riccati}) give \cite{coimbra} the classical Darboux
  transforms of $f$.  The condition (\ref{eq:initial condition}) then
  guarantees that $\hat f = f + T\invers$ has constant mean curvature.

We can now generalize the results on $\mu$--Darboux transforms for
constant mean curvature surfaces \cite{cmc} and Hamiltonian stationary
Lagrangians in \cite{hsl}: for a conformal immersion $f: M \to \R^4$
from a Riemann surface $M$ into 4--space, the Gauss map $\nu: M \to
\Gr_2(\R^4)$ is a map from $M$ into the Grassmannian of 2--planes in
$\R^4$. Identifying $\Gr_2(\R^4) = S^2\times S^2$ the Gauss map $\nu$
gives rise to two maps $N, R: M \to S^2$ satisfying
\[
*df = N df = - df R\,.
\] 
$N$ and $R$ are called the \emph{left} and \emph{right normal} of
$f$. From \cite{coimbra} we know that the $(1,0)$--part of $dN$ with
respect to $N$ is given by $(dN)' = - df H$ for some quaternion valued
function $H: M \to \H$ which satisfies $R H = HN$.

Examples of surfaces with harmonic left normal are constant mean
curvature surfaces in 3--space, minimal surfaces in 3--space, or
Hamiltonian stationary Lagrangian immersions in $\C^2=\R^4$. All
surfaces with harmonic left normal are constrained Willmore
\cite{hsl}. If a surface $f: M \to \R^4$ has harmonic left normal we
can associate again a family of flat connections $d_\lambda = d +
(\lambda-1)A\oz + (\lambda\invers-1)A\zo$ on $\trivial{} =\trivialC 2$
where $A$ is the Hopf field of the associated complex structure $J$ of
$N$. Note that the family $d_\lambda$ is trivial if and only if $f$ is
a minimal surface.

\begin{theorem}
\label{thm:darboux of harmonic left normal}
Let $f: M \to \R^4$ be a conformal immersion with harmonic left normal
$N: M \to S^2$ which is not a minimal surface,  so that $(dN)'=-df H$ with
non--trivial $H:M \to \H$. For $\mu\in\C_*$ let $
\varphi\in\Gamma(\ttrivial{})$ be a $d_\mu$--parallel section of the
associated family of flat connections of $N$. For $\mu\not=1$ put
$T=\frac 12(N\varphi(a-1)\varphi\invers + \varphi b \varphi \invers)$
with $a =\frac{\mu+\mu\invers}2, b = i \frac{\mu\invers-\mu}2$, and
\[
 \hat f= f + (HT)\invers
\] 
away from the (isolated) zeros of $H$.

Then the map $\hat f$ is either constant, or a (branched) conformal
immersion with harmonic left normal $\hat N = -T\invers N T$.
\end{theorem}
\begin{proof}
  As before (\ref{eq:dmu with a, b}) we have $d_\mu = d+*A(J(a-1)+b)$,
  so that with $(dN)'=-df H$ for a $d_\mu$--parallel section
  $\varphi\in\Gamma(\ttrivial{})$
\[
d\varphi = -df H T \varphi\,.
\]
Putting $\beta= HT\varphi$ this gives $0 = df\wedge d\beta$ which
implies $*d\beta =-Rd\beta$ by type arguments. In particular, $d\beta$
has only isolated zeros if $\beta$ is not constant \cite{Klassiker}.
From Theorem \ref{thm: mu darboux for gauss map} we see that $T$ has
no zeros so that $(HT)\invers$ is defined away from the zeros of $H$,
and
\[
d\hat f = df +d(H T)\invers = df + d(\varphi\beta\invers) = -
(HT)\invers d\beta \varphi\invers (HT)\invers
\]
shows that $\hat f$ is either constant, or a branched conformal
immersion with 
\[
*d\hat f = - (HT)\invers R (HT) d\hat f\,.
\]
Using $RH = HN$ we see that in the latter case $\hat f$ has left
normal
\[
\hat N = - T\invers N T\,.
\]
Theorem \ref{thm: mu darboux for gauss map} therefore shows that the
left normal $\hat N$ of $\hat f$ is harmonic.
\end{proof}

\begin{rem}
  Since $d_\mu$--parallel sections are holomorphic, the arguments in
  \cite{hsl} for the special case of Hamiltonian stationary
  Lagrangians show that $\hat f$ as defined in the above theorem is a
  generalized Darboux transform of $f$.  We call $\hat f$ a
  \emph{$\mu$--Darboux transform} of $f$ as it arises from a
  $d_\mu$--parallel section $\varphi$ for $\mu\in\C_*$.
\end{rem}

Similarly, a $\mu$--Darboux transformation is defined on the conformal
Gauss map of a (constrained) Willmore surface $f: M \to S^4$ and an
analogue of Theorem \ref{thm: mu darboux for gauss map} holds
\cite{willmore_harmonic}.

\section{Darboux transformation and simple factor dressing}

We show that the $\mu$--Darboux transformation and the simple factor
dressing of a harmonic map coincide. In particular, a $\mu$--Darboux
transform of a constant mean curvature surface $f: M \to \R^3$ is
given by a simple factor dressing of the Gauss map of the parallel
surface $g$ of $f$, and vice versa. This generalizes results for
classical Darboux transformations \cite{darboux_isothermic},
\cite{fran_epos}, \cite{inoguchi_kobayashi}. Moreover, since the
$\mu$--Darboux transformation is defined for all surfaces $f: M
\to\R^4$ with harmonic left normal, the simple factor dressing on the
harmonic left normal can thus also be given an interpretation on the
level of surfaces.

\begin{theorem}
\label{thm:result} 
  Let $N: M \to S^2$ be a non--trivial harmonic map. Then every $\mu$--Darboux
  transform of $N$ is given by a simple factor dressing, and vice
  versa.

  More precisely, if we denote by $d_\lambda$ the associated family of
  flat connections of $N$ and put $M_\mu =\varphi\C$ for a
  $d_\mu$--parallel section $\varphi\in\Gamma(\ttrivial{})$,
  $\mu\in\C_*$, then the simple factor dressing $\hat N$ of $N$ with
  respect to $M_\mu$ is the $\mu$--Darboux transform of $N$ with
  respect to $\varphi$, that is
\[
\hat N = T\invers N T
\]
with $T = \frac 12(N\varphi(a-1)\varphi\invers + \varphi b\varphi
\invers)$ and $a=\frac{\mu+\mu\invers}2, b = i
\frac{\mu\invers-\mu}2$.
\end{theorem}

\begin{proof}
  In this proof we adapt the arguments in \cite{quintino} to the case
  of harmonic maps $N: M \to S^2$, and generalize her setting from
  $\mu\in\R_*\cup S^1$ to the general case $\mu\in\C_*$: As before, we
  denote by $\hat z = \varphi z\varphi\invers$ for $z\in\C$, and
  recall that $\hat a^2 + \hat b^2=1$.  Let $J$ be the complex
  structure of $N$ and $E$ the $+i$ eigenspace of $J$. Putting $\rho
  =\frac{1-a}2$ and $\hat T = T \hat\rho\invers$ we first show that
  the $+i$ eigenspace of the complex structure $\hat J$ of $\hat N =
  T\invers N T$ is given by $\hat E = \hat T E$: the equation
  (\ref{eq:initial condition}) shows $(\hat T + N)^2 = -1 + \hat
  \rho\invers$, that is,
\[
\hat T^2 + \hat T N + N \hat T = \hat \rho\invers\,.
\]
From this we see  that $N$ commutes with
\[
\hat \rho  T^{-2} = 1 + N \hat T\invers + \hat T\invers N
\]
and thus $[T^2\hat \rho\invers, N]=0$.  For $\hat \phi = \hat T\phi$,
$\phi\in E$ we therefore obtain
\[
\hat N \hat \phi = \hat T N \phi = \hat \phi i\,,
\]
and $\hat E$ is the $+i$ eigenspace of $\hat J$.  

Since $\hat N$ is completely determined by the $+i$ eigenspace of
$\hat J$ it is enough to show by Remark \ref{rem:dressing complex
  structure} that $\hat E = r_\infty E$ where $r_\lambda = \pi_\mu
\circ \gamma_\lambda + \pi_\mu^\perp$.  Here $\pi_\mu$ and
$\pi_\mu^\perp$ are the projections onto $M_\mu$ and $M_\mu^\perp$
respectively, and $\gamma_\lambda = \frac{1-\bar\mu\invers}{1-\mu} \
\frac{\lambda-\mu}{\lambda-\bar\mu\invers}$.  We first observe that
$a-1= \frac{ \mu\invers(\mu-1)^2}2$ and
$b=i\frac{\mu\invers(1-\mu^2)}2$ so that
\[
\frac{b}{1-a} = i \frac{\mu+1}{\mu-1}\,.
\]
Since $\varphi\in \Gamma(M_\mu)$ we have $r_\infty \varphi = \varphi
\frac{1-\bar\mu\invers}{1-\mu}$ and $r_\infty(\varphi j) = \varphi j$,
so that (\ref{eq:initial condition}) shows
\[
(\hat T + N - I)\varphi = \varphi  \frac{2i}{\mu-1} =- (r_\infty)\frac{2I}{1-\bar\mu\invers}\varphi
\]
and 
\[
(\hat T + N - I)(\varphi j) = -\varphi j \frac {2i}{1-\bar\mu\invers} =
-(r_\infty)\frac{2I}{1-\bar\mu\invers}\varphi j\,,
\]
in other words,
\[
\hat T + N - I = - r_\infty \circ \frac{2I}{1-\bar\mu\invers}\,.
\]
Finally, for $\phi\in E$ we have $(N-I)\phi=0$ since $E$ is the $+i$ eigenspace of $J$, and thus
\[
\hat T\phi =- r_\infty \phi \frac{2i}{1-\bar\mu\invers}\,.
\]
This proves that $\hat T E = r_\infty E$, and thus $\hat N$ is the
simple factor dressing of $N$ by $r_\lambda$.

\end{proof}

As an immediate consequence of Theorem \ref{thm:result} and Theorem
\ref{thm:mu-darboux} simple factor dressing and the $\mu$--Darboux
transformation are essentially the same for constant mean curvature
surfaces:

\begin{theorem}
  The Gauss map of a $\mu$--Darboux transform $\hat f$ of a constant
  mean curvature surface $f: M \to \R^3$ is a simple factor dressing
  of the Gauss map of the parallel surface $g=f+N$ of $f$, and vice versa.
\end{theorem}

More generally, if $f$ is a surface with harmonic left normal $N$,
then Theorem \ref{thm:result} and Theorem \ref{thm:darboux of harmonic
  left normal} show that a simple factor dressing of $N$ is induced by a
transformation on the surface $f$:

\begin{theorem}
  Let $f:M \to\R^4$ be a conformal immersion with harmonic left normal
  $N: M \to S^2$ which is not a minimal surface. Then a simple factor
  dressing of $-N$ is the left normal of a $\mu$--Darboux transform of
  $f$, and vice versa.
\end{theorem}


\bibliographystyle{alpha}
\bibliography{doc}

\newcommand{\etalchar}[1]{$^{#1}$}
\begin{thebibliography}{BFL{\etalchar{+}}02}

\bibitem[BFL{\etalchar{+}}02]{coimbra}
F.~Burstall, D.~Ferus, K.~Leschke, F.~Pedit, and U.~Pinkall.
\newblock {\em Conformal Geometry of Surfaces in ${S}^4$ and Quaternions}.
\newblock Lecture Notes in Mathematics, Springer, Berlin, Heidelberg, 2002.

\bibitem[BLPP08]{conformal_tori}
C.~Bohle, K.~Leschke, F.~Pedit, and U.~Pinkall.
\newblock Conformal maps from a 2--torus to the 4--sphere.
\newblock arXiv:0712.2311v1, 2008.

\bibitem[Bob91]{sym_bob}
A.~Bobenko.
\newblock Constant mean curvature surfaces and integrable equations.
\newblock {\em Russ. Math. Surv. 40}, pages 1--45, 1991.

\bibitem[Bur06]{fran_epos}
F.~E. Burstall.
\newblock Isothermic surfaces: conformal geometry, {C}lifford algebras and
  integrable systems.
\newblock In {\em Integrable systems, geometry, and topology}, volume~36 of
  {\em AMS/IP Stud. Adv. Math.}, pages 1--82. Amer. Math. Soc., Providence, RI,
  2006.

\bibitem[CLP10]{cmc}
E.~Carberry, K.~Leschke, and F.~Pedit.
\newblock Darboux transforms and spectral curves of constant mean curvature
  surfaces revisited.
\newblock Preprint, 2010.

\bibitem[DK05]{dorfmeister_kilian}
J.~Dorfmeister and M.~Kilian.
\newblock Dressing preserving the fundamental group.
\newblock {\em Diff. Geom. Appl. 23}, pages 176--–204, 2005.

\bibitem[FLPP01]{Klassiker}
D.~Ferus, K.~Leschke, F.~Pedit, and U.~Pinkall.
\newblock Quaternionic holomorphic geometry: {P}l\"ucker formula, {D}irac
  eigenvalue estimates and energy estimates of harmonic 2-tori.
\newblock {\em Invent. math., Vol. 146}, pages 507--593, 2001.

\bibitem[Hit90]{hitchin-harmonic}
N.~Hitchin.
\newblock Harmonic maps from a $2$-torus to the $3$-sphere.
\newblock {\em J. Differential Geom., Vol 31 (3)}, pages 627--710, 1990.

\bibitem[HJP97]{darboux_isothermic}
U.~Hertrich-Jeromin and F.~Pedit.
\newblock {R}emarks on the {D}arboux transfoms of isothermic surfaces.
\newblock {\em Doc. Math. J. DMV, Vol 2}, pages 313--333, 1997.

\bibitem[IK05]{inoguchi_kobayashi}
J.~Inoguchi and S.~Kobayashi.
\newblock Characterizations of {B}ianchi--{B}\"acklund transformations of
  constant mean curvature surfaces.
\newblock {\em Internat. J. Math. 16 (2)}, pages 101--110, 2005.

\bibitem[Les10]{willmore_harmonic}
K.~Leschke.
\newblock Harmonic map methods for {W}illmore surfaces.
\newblock arXiv:1003.3371, 2010.

\bibitem[LR10]{hsl}
K.~Leschke and P.~Romon.
\newblock Darboux transforms and spectral curves of {H}amiltonian stationary
  {L}agrangian tori.
\newblock {\em Calc. Var. PDE, Vol 38 (1)}, pages 45--74, 2010.

\bibitem[Qui08]{quintino}
A.~Quintino.
\newblock {\em Constrained {W}illmore Surfaces: Symmetries of a Moebius
  Invariant Integrable System}.
\newblock PhD thesis, University of Bath, 2008.
\newblock arXiv:0912.5402.

\bibitem[RV70]{ruh_vilms}
E.~Ruh and J.~Vilms.
\newblock The tension field of the {G}auss map.
\newblock {\em Trans. Am. Math. Soc., Vol 149}, pages 569--573, 1970.

\bibitem[TU00]{terng_uhlenbeck}
C.~Terng and K.~Uhlenbeck.
\newblock B{\"a}cklund transformations and loop group actions.
\newblock {\em Comm. Pure and Appl. Math LIII}, pages 1--–75, 2000.

\bibitem[Uhl89]{uhlenbeck}
K.~Uhlenbeck.
\newblock {H}armonic maps into {L}ie groups (classical solutions of the chiral
  model).
\newblock {\em J. Diff.\ Geom., Vol 30}, pages 1--50, 1989.

\end{thebibliography}

\end {document}